\makeatletter

\def\eqalign#1{\,\vcenter{\openup\jot\m@th
  \ialign{\strut\hfil$\displaystyle{##}$&$\displaystyle{{}##}$\hfil
      \crcr#1\crcr}}\,}
\def\eqalignno#1{\displ@y \tabskip\@centering
  \halign to\displaywidth{\hfil$\displaystyle{##}$\tabskip\z@skip
    &$\displaystyle{{}##}$\hfil\tabskip\@centering
    &\llap{$##$}\tabskip\z@skip\crcr
    #1\crcr}}
\def\leqalignno#1{\displ@y \tabskip\@centering
  \halign to\displaywidth{\hfil$\displaystyle{##}$\tabskip\z@skip
    &$\displaystyle{{}##}$\hfil\tabskip\@centering
    &\kern-\displaywidth\rlap{$##$}\tabskip\displaywidth\crcr
    #1\crcr}}

\makeatother

\newdimen\pixel \pixel=.00333333 true in

\documentclass[11pt]{article}
\usepackage{epsfig}
\usepackage{fullpage}

\makeatletter
\def\bigpar{\bigbreak\@afterindentfalse\@afterheading\ignorespaces}
\def\medpar{\medbreak\@afterindentfalse\@afterheading\ignorespaces}
\def\smallpar{\smallbreak\@afterindentfalse\@afterheading\ignorespaces}

\newlength{\saveindent}
\newenvironment{proof}%
      {\bigpar{\bf Proof:}\ 
             \setlength{\saveindent}{\parindent} 
                       \ignorespaces}%
      {\stopproof\ignorespaces\bigbreak \setlength{\parindent}{\saveindent}}

      {\bigpar{\bf Proof:}%
             \setlength{\saveindent}{\parindent} 
                       \,\ignorespaces}%
      {\ignorespaces\bigbreak \setlength{\parindent}{\saveindent}}

      {\bigpar{\bf Proof:}\ %
             \setlength{\saveindent}{\parindent} 
                       \ignorespaces}%
      {\ignorespaces\bigbreak \setlength{\parindent}{\saveindent}}

      {\bigpar{\bf#1:}\ %
             \setlength{\saveindent}{\parindent} 
                       \ignorespaces}%
      {\stopproof\ignorespaces\bigbreak \setlength{\parindent}{\saveindent}}

\newenvironment{remark}%
      {\smallpar{\bf Remark:}\ 
                       \ignorespaces}%
      {\stopproof\ignorespaces\medbreak \setlength{\parindent}{\saveindent}} 

\newenvironment{remark*}%
      {\smallpar{\bf Remark:}\ 
                       \ignorespaces}%
      {\ignorespaces\medbreak \setlength{\parindent}{\saveindent}} 

      {\smallpar{\bf Remarks:}\ 
                       \ignorespaces}%
      {\stopproof\ignorespaces\medbreak \setlength{\parindent}{\saveindent}}

\newenvironment{remarks*}%
      {\smallpar{\bf Remarks:}\ 
                       \ignorespaces}%
      {\ignorespaces\medbreak \setlength{\parindent}{\saveindent}}

      {\smallpar{\bf Remark:}\ %
                       \ignorespaces}%
      {\stopproof\ignorespaces\medbreak \setlength{\parindent}{\saveindent}}

      {\smallpar{\bf Remarks:}\ %
                       \ignorespaces}%
      {\stopproof\ignorespaces\medbreak \setlength{\parindent}{\saveindent}}
\makeatother

\newtheorem{theorem}{Theorem}
\newtheorem{lemma}[theorem]{Lemma}

\newtheorem{proposition}[theorem]{Proposition}
\newtheorem{corollary}[theorem]{Corollary}

\newtheorem{example}{Example}

\def\begex{\begin{example}\parindent=0pt \rm}
\def\endex{\end{example}}
\def\square{\vbox{\hrule height.2pt\hbox{\vrule width.2pt height5pt \kern5pt
                                   \vrule width.2pt} \hrule height.2pt}}
\def\stopproof{\hfill \square \smallskip}

\def \W {{\cal L}}

\def \G {{Q}}

\def\half{{\textstyle{1\over2}}}

\def \vertices {{V}}

\def \r {{\bf R}}
\def \var {{ \rm var }}

\def \gt {{\tilde G}}

\def \dS {{\partial S}}

\def\r|{{\Bigr\vert}}
\def\l|{{\Bigl\vert}}

\def \R {{\bf R}}

\def \bad {{\cal B}}

\def \good {{\cal G}}

\def\phi {\Phi}

\def\e{\epsilon}
\def\twosubs#1#2{\scriptstyle#1\atop\scriptstyle#2}

\def \Z {{\bf Z}}

\def\config {{\it S}_n}

\def\varepsilon{\mathchar"122 }

\def \Epsilon {{\cal E}}

\def\e {{ \bf {E}}}

\def\Sq{{\cal S}_q}
\def\Sq-{{\cal S}_{q-1}}

\def \IP {{\rm IP}}

\def \RW {{\rm RW}}

\newcommand{\be}{\begin{equation}}
\newcommand{\ee}{\end{equation}}
\newcommand{\lab}{\label}
\begin{document}
\title{Spectral gap for the interchange process in a box}
\author{
{\sc Ben Morris}\thanks{Department of Mathematics,
University of California, Davis CA 95616.  
\newline Email:
{\tt morris@math.ucdavis.edu}.  
Partially supported by Sloan Fellowship and 
NSF grant DMS-0707144.}
}
\maketitle
\thispagestyle{empty}
\begin{abstract}
\noindent
We show that the 
spectral gap for the interchange process (and the
symmetric exclusion process) in a
$d$-dimensional box 
of side length $L$ is asymptotic to $\pi^2/L^2$.
This gives more evidence in favor of 
Aldous's conjecture that in any graph the spectral gap for the interchange 
process is the same as the spectral gap for a corresponding 
continuous-time random walk. Our proof uses a technique that is similar to that
used by Handjani and Jungreis, who proved that Aldous's conjecture holds when 
the graph is a tree. 
\end{abstract}

\section{Introduction}
\subsection{Aldous's conjecture}
This subsection is taken (with minor alterations) from Davis Aldous's web page.
Consider an $n$-vertex graph $G$ which is connected and undirected. 
Take $n$ particles labeled $1,2,...,n$. In a  configuration, there is
one 
particle at each vertex. The {\it interchange process} is the following 
continuous-time Markov chain on configurations. For each edge $(i,j)$, 
at rate $1$ the particles at vertex $i$ and vertex $j$ are interchanged.

The interchange process is reversible, and its stationary distribution
is 
uniform on all $n!$ configurations. There is a spectral gap 
$\lambda_{\IP}(G) > 0$, which is the absolute value of the largest non-zero eigenvalue of
the transition rate matrix. If instead we just watch a single
particle, 
it performs a continuous-time random walk on $G$ (hereafter
referred to simply as ``the continuous-time random walk on $G$''),
which is also
reversible 
and hence has a spectral gap $\lambda_{\RW}(G) > 0$. 
Simple arguments 
(the contraction principle) show $\lambda_{\IP}(G) \leq \lambda_{\RW}(G) $.
\\
\\
{\bf Problem}. Prove $\lambda_{\IP}(G) = \lambda_{\RW}(G) $ for all
$G$.\\
\\
{\bf Discussion.} Fix $m$ and color particles $1,2,...., m$ red. 
Then the red particles in the interchange process behave as the usual
exclusion process (i.e., $m$ particles performing the continuous-time
random walk on $G$, but with moves that take two particles to the 
same vertex suppressed).
But in the finite setting, the interchange process
seems more natural. 
\subsection{Results}  
Aldous's conjecture has been proved in the case where $G$ is a tree
\cite{hj} and in the case where $G$ is the complete 
graph \cite{ds}; see also \cite{ns}.  
In this note we prove an asymptotic version
of Aldous's conjecture for
$G$ a box in $\Z^d$. We show that if $B_L$ denotes a box of 
side length $L$ in $\Z^d$
then 
\[
{\lambda_{\IP}(B_L) \over \lambda_{\RW}(B_L)} \to 1,
\]
as $L \to \infty$. \\
\\ 
\begin{remark}
After completing a draft of this paper, I learned that Starr and Conomos had
recently obtained the same 
result  
(see \cite{starr}). Their proof uses a similar approach, although the present paper 
is somewhat shorter. 
\end{remark}
 
{\bf Connection to simple exclusion. }
Our result gives a bound on the 
spectral gap for the exclusion process. The exclusion process 
is a widely studied Markov chain, with 
connections to
card shuffling \cite{wilson, bbhm}, 
statistical mechanics \cite{kov, q, cm, thomas}, and a variety of other
processes (see e.g., \cite{liggett, fill}); it
has been one of the major examples behind 
the study of convergence rates for Markov chains (see, e.g., \cite{fill,
  ct, wilson, bbhm}). 
Our result implies that the spectral gap for the symmetric
exclusion process in  $B_L$ is 
asymptotic to $\pi^2/L^2$.
The problem of bounding the spectral gap for simple 
exclusion was studied in
Quastel \cite{q} and a subsequent independent paper of
Diaconis and Saloff-Coste \cite{ct}.
Both of these papers 
used a comparison to Bernoulli-Laplace diffusion 
(i.e., the exclusion process in the complete graph) to 
obtain a bound of order $1/dL^2$. 
Diaconis and Saloff-Coste explicitly wondered whether 
the factor $d$ in the denominator is necessary; in the 
present paper we show that it is not.
\section{Background}
Consider a continuous-time Markov chain on a finite state space $W$
with a symmetric transition rate matrix $Q(x,y)$.  The
spectral gap is the minimum value of $\alpha >0$ such that \be
\label{eigen}
\G f = -\alpha f,
\ee
for some $f: W \to \R$. The spectral gap governs the asymptotic rate of convergence 
to the stationary distribution. 
Define
\[
\Epsilon(f,f) = {1 \over 2|W|} \sum_{x, y \in W} (f(x) - f(y))^2 Q(x,y),
\]
and define
\[
\var(f) = {1 \over |W|} \sum_{x \in W}  (f(x) - \e(f))^2,
\]
where 
\[
\e(f) = {1 \over |W|} \sum_{x \in W} f(x).
\]
If $f$ is a function that 
satisfies $\G f = - \lambda f$ for some $\lambda > 0$, then
\be
\label{fund}
\lambda = { \Epsilon(f, f) \over \var(f)}.
\ee
Furthermore, 
if $\alpha$ is the spectral gap then
for any non-constant $f: W \to \R$ we have
\be
\label{variational}
{\Epsilon(f,f) \over 
\var(f)} \geq \alpha. 
\ee
Thus the spectral gap can be obtained by minimizing the left hand 
side of (\ref{variational}) over all non-constant functions $f: W \to \R$.

\section{Main result} 
Before specializing to the interchange process, we first prove 
a general proposition relating the
eigenvalues of a certain function of a Markov chain to 
the eigenvalues of the Markov chain itself.
Let $X_t$ be a  continuous-time Markov chain on a finite state space $W$
with a symmetric transition rate matrix $Q(x,y)$. 
Let $T$ be another space and 
let $g: W \to T$ be a function on $W$ such that
if $g(x) = g(y)$ and $U = g^{-1}(u)$ for some $u$, then
$\sum_{u' \in U} \G(x, u') = \sum_{u' \in U} \G(y, u')$. 
Note that $g(X_t)$ is a Markov chain. 
Let $W'$ denote the collection of 
subsets of $W$ of the 
form $g^{-1}(u)$ for some $u \in T$. 
We can identify the states of $g(X_n)$ with 
elements of $W'$.
Let $\G'$ denote the transition 
rate matrix for $g(X_n)$. Note that
if $U, U' \in W'$, with 
$U = g^{-1}(u)$ for some $u \in T$ and
$U \neq U'$, then 
$\G'(U,U') = \sum_{y \in U'} \G(u,y)$.  

We shall need the following proposition, which generalizes Lemma
2 of \cite{hj}.
\begin{proposition}
\label{ezlemma}
Let $X_t$, $g$ and $\G'$ be as defined above.
Suppose $f: W \to \R$ is an eigenvector of $\G$ 
with corresponding eigenvalue $-\lambda$ and
define $h: W' \to \R$ by $h(U) = 
\sum_{x \in U} f(x)$.
Then $\G' h = -\lambda h$. That is, either $h$ is an eigenvector 
of $\G'$ with corresponding eigenvalue $-\lambda$, or $h$ is
identically 
zero.
\end{proposition}
\begin{proof}
Note that for all $U' \in W'$ we have
\begin{eqnarray*}
(\G' h) (U') &=&
\sum_{U \in W'} h(U) \G'(U, U') \\
&=& \sum_{U \in W'} \sum_{x \in U} f(x) \sum_{y \in U'}
\G(x, y) \\
&=& \sum_{y \in U'} (\G f) (y) \\
&=& -\lambda \sum_{y \in U'} f(y) \\
&=& -\lambda h(U'),
\end{eqnarray*}
so $\G' h = -\lambda h$.
\end{proof}
\noindent

The following Lemma is a weaker version of Aldous's conjecture.
The proof is 
similar to the proof of Theorem 1 in \cite{hj}. 
\\
\begin{lemma}
\label{maintheorem}
Let $G$ be a connected, undirected graph with vertices
labeled $1, \dots, n$.
For $2 \leq k \leq n$ let $G_k$ be the subgraph of $G$ 
induced by the vertices $1, 2, \dots, k$. Let $\lambda_{\RW}(G_k)$
be the spectral gap for the continuous-time random walk on 
$G_k$, and define $\alpha_k = \min_{2 \leq j \leq k}
\lambda_{\RW}(G_j)$. Then
\[
\lambda_{\IP}(G) \geq \alpha_n.
\]
\end{lemma}
\begin{proof}
Our prooof will be by induction on the number of vertices $n$.
The base case $n=2$ is trivial, so assume $n > 2$. 
Let $W$ and $\G$ be the state space and transition rate matrix, respectively, for 
the interchange process on $G$.
Let $f: W \to \R$
be a function that satisfies
$\G f = - \lambda f$.
We shall show that $\lambda \geq \alpha_n$.  
Note that a configuration of the interchange process can 
be identified with a permutation $\pi$ in $\config$, where
if particle $i$ is in vertex $j$, then $\pi(i) = j$. 
For positive integers $m$ and $k$
with $m,k \leq n$, we 
write
$f( \pi(m) = k)$ for 
\[
\sum_{\pi: \pi(m) = k} f(\pi).
\]
We consider two cases.\\
\\
{\bf Case 1: } For some $m$ and $k$ we have $f(\pi(m) = k) \neq 0$.
Define $h: V \to \R$ by $h(j) = f(\pi(m) = j)$.
Then $h$ is not identically  zero,
and using
Proposition \ref{ezlemma} 
with $g$ defined by $g(\pi) = \pi(m)$
gives that 
if $\G'$ is the transition rate matrix for continuous time 
random walk on $G$, then $\G' h = - \lambda h$.
It follows that $\lambda$ is an eigenvalue of $\G'$ and 
hence $\lambda \geq \lambda_{\RW}(G) = \alpha_n$. \\
\\
{\bf Case 2:} For all 
$m$ and $k$ we have $f(\pi(m) = k) = 0$.
Define the {\it suppressed process} as the interchange process with 
moves involving vertex $n$ suppressed. That is, the Markov chain with the following transition rule: \\
\\
For every edge $e$ not incident to $n$, at rate $1$ 
switch the particles at the endpoints of $e$. \\
\\
For $1 \leq k \leq n$, let $W_k = \{ \pi \in W: \pi^{-1}(n) = k\}$. Note that
the $W_k$ are the 
irreducible classes 
of the suppressed process, and 
that for each $k$
the restriction of the suppressed process to $W_k$ can be identified with the interchange 
process on $G_{n-1}$. 
For $k$ with $1 \leq k \leq n$, define
\[
\Epsilon_k(f,f) = {1 \over 2(n-1)!} \sum_{\pi_1, \pi_2 \in W_k} (f(\pi_1) - f(\pi_2))^2 Q(\pi_1,\pi_2),
\]
and define
\[
\var_k(f) = {1 \over (n-1)!} \sum_{\pi \in W_k}  f(\pi)^2.
\]
(Note that for every $k$ we have
$\sum_{\pi \in W_k} f(x) = 0$.)

By the induction hypothesis, the spectral gap for the interchange process 
on $G_{n-1}$ is at least $\alpha_{n-1}$. Hence
for every $k$ with $1 \leq k \leq n$ 
we have
\[
\Epsilon_k(f,f) \geq \alpha_{n-1} \var_k(f) \geq \alpha_n \var_k(f).
\]
It follows that
\begin{eqnarray}
n! \Epsilon(f,f) &\geq& \half \sum_{k=1}^n 
\sum_{\pi_1, \pi_2 \in W_k} (f(\pi_1) - f(\pi_2))^2 \G(\pi_1, \pi_2) \\
&=&  \sum_{k=1}^n (n-1)! \Epsilon_k(f,f) \\
\label{tq}
&\geq& \sum_{k=1}^n  \alpha_{n} 
(n-1)! \var_k(f) \\
&=&\alpha_{n} \sum_{k=1}^n \sum_{\pi \in W_k}
f(\pi)^2 \\
&=& \alpha_{n} n! \var(f).
\end{eqnarray}
Combining this with 
equation (\ref{fund})
gives
$\lambda \geq \alpha_n$.
\end{proof}
\begin{remark}
Theorem \ref{maintheorem} is optimal if the vertices are 
labeled in such a way that 
$\lambda_{\RW}(G_k)$ 
is  nonincreasing in $k$, 
in which case it gives $\lambda_{\IP}(G) = 
\lambda_{\RW}(G)$. Since any tree can be built up from smaller trees
(with larger spectral gaps), we recover the 
result proved in \cite{hj} that
$\lambda_{\IP}(T) = 
\lambda_{\RW}(T)$ if $T$ is a tree. 
\end{remark}

Our main application of Lemma \ref{maintheorem} is the 
following asymptotic version of Aldous's conjecture in the special 
case where $G$ is a box in $\Z^d$. 

\begin{corollary}
Let $B_L = \{0, \dots, L\}^d$ be a box of side length $L$ in 
$\Z^d$. Then the spectral gap for the interchange process on $B_L$ is
asymptotic to $\pi^2/L^2$.
\end{corollary}
\begin{proof}
In order to use Lemma \ref{maintheorem} we need to
label the vertices of $B_L$ in some way.
Our goal is to label in such a way that 
for every $k$ the quantity 
$\lambda_{\RW}(G_k)$ 
(i.e., the spectral gap corresponding to
the subgraph of $B_L$ induced by the vertices
$1, \dots, k$) 
is not too much smaller than
$\lambda_{\RW}(B_L)$. 
So our task is to build $B_L$, one vertex at a time, in such
a way that the spectral gaps of the intermediate graphs 
don't get too small. 

We shall build $B_L$ by inductively building $B_{L-1}$ and 
then building $B_L$ from $B_{L-1}$.
Since $\lambda_{\RW}(B_L) \downarrow 0$, it is enough to show that
\[
{\beta_L \over \lambda_{\RW}(B_L)} \to 1,
\]
where $\beta_L$ is the minimum spectral gap for any intermediate
graph between $B_{L-1}$ and $B_{L}$.

 For a graph $H$, let $\vertices(H)$ denote the set of vertices in $H$.
For $j \geq 0$, let $\W_j = \{0, \dots, j\}$ be the line 
graph with $j + 1$ vertices. 
Define $\gamma_L = \lambda_{\RW}(\W_L)$. It is well known that
$\gamma_L$ is decreasing in $L$ and asymptotic to $\pi^2/L^2$ 
as $L \to \infty$. 
It is also well known that if $H$ and $H'$ are graphs and $\times$ 
denotes Cartesian product, then 
$\lambda_{\RW}(H \times H') = 
\min(\lambda_{\RW}(H),\lambda_{\RW}(H'))$. 
Since $B_L = \W_L^d$, 
it follows that 
$\lambda_{\RW}(B_L) = \gamma_L$.

 We construct $B_{L}$ from $B_{L-1}$ 
using intermediate graphs
$H_0, \dots, H_d$, where for $k$ with $1 \leq k \leq d$ we define
$H_k = \W_L^{k} \times 
\W_{L-1}^{d - k}$. Note that
$H_0 = B_{L-1}$ and $H_d = B_{L}$. We obtain
$H_k$ from $H_{k-1}$  
by adding vertices to lengthen $H_{k-1}$ 
by one unit in direction $k$.  The order in which the vertices 
in $\vertices(H_k) - \vertices(H_{k-1})$ are added is 
arbitrary. 

 Fix $k$ with $1 \leq k \leq d$, and define $G' = G'(L,k)$ as 
follows.
Let 
\[
V' = \vertices(H_k),
\hspace{.4 in}
E' = \{(u, v) : \mbox{either $u$ or $v$ is a vertex in $H_{k-1}$}\},
\]
and let $G' = (V', E')$. 
It is well known and easily shown that if $H$ is a graph, then adding edges to $H$ 
cannot decrease $\lambda_{\RW}(H)$, nor can removing pendant edges. 
Since each intermediate graph $\gt$ 
between $H_{k-1}$ and $H_k$ can be obtained from $G'$ by adding edges and
removing pendant edges,  it follows that for 
any such graph $\gt$ we have $\lambda_{\RW}(\gt) \geq \lambda_{\RW}(G')$. 
Thus, it is enough to bound $\lambda_{\RW}(G')$ from below.
We shall show that for any $\epsilon > 0$ we
have $\lambda_{\RW}(G'(L,k)) \geq (1-\epsilon) \gamma_L$ if $L$ 
is sufficiently large.

 Let $e_k$ be the unit vector in direction $k$.
Let 
\[
S = \vertices(H_{k-1}); \hspace{.4 in}  
\dS = V' - S.
\]
Let $X_t$ be the continuous-time random walk on 
$G'$, with transition rate matrix $\G$.
Fix 
$f: V' \to \R$ with $\G f = -\lambda f$ for 
some $\lambda > 0$. For $x \in \Z^d$, let $g_k(x)$ denote the component
of
$x$ in the $k$th coordinate. Note that $g_k(X_t)$ is the 
continuous-time random walk on $\W_L$. 
Let $\G'$ be the transition rate matrix for 
$g_k(X_t)$.
Proposition \ref{ezlemma} 
implies that if $h: \{0, \dots, L\} \to \R$
is defined by 
$h(j) = 
\sum_{\twosubs{x \in V'}{g_k(x) = j}} f(x)$,
then
$\G' h = - \lambda h$.
Thus if $\lambda < \gamma_L$, then $g$ is identically zero
and hence $\sum_{x \in S} f(x) = 0$. 
Define
\[
\Epsilon(f,f) = {1 \over 2|V'|} \sum_{x, y \in V'} (f(x) - f(y))^2 Q(x,y),
\]
and let $\Epsilon_S(f,f)$ be defined 
analogously, but with only vertices in $S$ included in the sum. 
Note that $\Epsilon(f,f) \geq \Epsilon_S(f,f)$. 
Since $\sum_{x \in S} f(x) = 0$, we have
\be
\label{eqb}
{ \Epsilon_S(f,f) \over \sum_{x \in S} f(x)^2}
\geq \lambda_{\RW}(H_{k-1}) \geq \gamma_L,
\ee
where the second inequality follows from the fact that
$H_{k-1}$
is a Cartesian product of $d$ graphs, each of which is either $\W_{L-1}$ 
or $\W_L$.

Fix $\epsilon > 0$ and let 
$M$ be a positive integer large enough so that
$(1 - 4M^{-1})^{-1} \leq (1 - \epsilon)^{-1/2}$.
For each $x \in \dS$, say that $x$ 
is {\it good} if there is a $y \in S$ such that $x = y + ie_k$ for some 
$i \leq M$ and $|f(y)| \leq |f(x)|/2$. Otherwise say that $x$ is
{\it bad}. Let $\good$ and $\bad$ denote the set of good and bad
vertices, respectively, in $\dS$. 
Note that if $x$ is bad and $M \leq L$ then
$f(x)^2 \leq {4 \over M} \sum_{j = 1}^M f(x - j e_k)^2$. Summing this over 
bad $x$ gives
\be
\lab{eqc}
\sum_{x \in \bad}
f(x)^2 \leq {4 \over M} \sum_{x \in V'} f(x)^2
\ee
Note that
if $x$ is good, then there must be an 
$x' \in S$ of the form $x - ie_k$ 
such that 
$|f(x') - f(x' + e_k)| > f(x)/2M$. 
It follows that 
\be
\lab{eqa}
{\Epsilon(f,f) \over \sum_{x \in \good} f(x)^2} 
\geq 1/4M^2.
\ee
Since $V' = S \cup \bad \cup \good$, combining 
equations (\ref{eqa}), (\ref{eqb}) and
(\ref{eqc}) gives
\[
\sum_{x \in V'} f(x)^2
\leq (\gamma_L^{-1} + 4M^{2}) \Epsilon(f,f) + 4M^{-1}
\sum_{x \in V'} f(x)^2,
\]
and hence
\be
\lab{starrr}
\sum_{x \in V'} f(x)^2
\leq 
(1 - 4M^{-1})^{-1} (\gamma_L^{-1} + 4M^{2}) \Epsilon(f,f).
\ee
Recall that 
$(1 - 4M^{-1})^{-1} \leq (1 - \epsilon)^{-\half}$, 
and note that
since $\gamma_L \to 0$ as $L \to \infty$, 
we have
$\gamma_L^{-1} + 4M^{2} \leq (1 -\epsilon)^{-\half} \gamma_L^{-1}$
for sufficiently large $L$. Combining this with equation (\ref{starrr})
gives
\[
{\Epsilon(f,f) \over 
\sum_{x \in V'} f(x)^2} \geq (1 - \epsilon) \gamma_L,
\]
for sufficiently large $L$. 
It follows that $\lambda_{\RW}(G') \geq (1 - \epsilon) \gamma_L$ for sufficiently 
large $L$ and so the proof is complete.
\end{proof}


\begin{thebibliography}{99}
\bibitem{bbhm} Benjamini, I., Berger, N., Hoffman, C., and Mossel, E. 
Mixing times of the biased card shuffling and the asymmetric exclusion
process. Preprint. 

\bibitem{cm} Cancrini, N.~and Martinelli, F. (2000)
On the spectral gap of 
Kawasaki dynamics under a mixing condition revisited. {\em Journal of
  Math.~Phys.~41}, 1391--1423.

\bibitem{ct}
Diaconis, P. and Saloff-Coste, L. (1993).
Comparison Theorems for reversible Markov chains.  
{\em  Ann.~Appl.~Prob.~3}, 696--730.

\bibitem{dsc}
Diaconis, P. and Saloff-Coste, L. (1996).
Logarithmic Sobolev Inequalities for Finite Markov Chains,
{\em  Ann. Appl. Prob. 6}, 695--750.

\bibitem{ds}
Diaconis, P. and Shahshahani, M. (1981).
Generating a random permutation with random transpositions.
{\em  Z. Wahrsch. Verw. Geb. 57}, 159--179.






\bibitem{fill} Fill, J.~(1991). Eigenvalue bounds on convergence to
  stationarity for nonreversible Markov chains with an application to
  the exclusion processes. {\em Ann.~Appl.~Prob.~1}, 62--87. 

\bibitem{hj} Handjani, S.~and Jungreis, D.(1996). 
Rate of convergence for shuffling cards by transpositions.
{\em J.~Theor.~Prob.~9}, 983--993. 


\bibitem{kov} Kipnis, C., Olla, S., and Waradhan,
  S.~(1989). Hydrodynamics and large deviations for simple exclusion
  processes. {\em Comm.~Pure Appl.~Math.~42}, 115--137. 

\bibitem{leeyau} Lee, T.-Y.~and Yau, H.-T.~(1998). 
Logarithmic Sobolev inequality for some models of random walk.
{\em Ann.~Prob.~26}, 1855--1873.

\bibitem{liggett} Liggett, T.M. {\em Interacting Particle Systems.}
  Springer, 1985. Grundlehren der mathematischen Wissenschaften; 276.


\bibitem{ly} Lu, S.~and Yau, H-T (1993).
Spectral gap and logarithmic Sobolev inequality for Kawasaki and 
Glauber dynamics. {\em Comm.~Math.~Phys.~156}, no.~2, 399-433.

\bibitem{ns} Nachtergaele, B.~and Starr, S. 
Ordering of energy levels in Heisenberg models and 
applications. In Joachim Asch and Alain Joye (Eds.),
``Mathematical Physics of Quantum Mechanics:
Selected and refereed lectures from QMath9'', {\it Lecture notes in Physics,}
{\bf Vol 690.} Springer, 2006.


\bibitem{q} Quastel, J (1992).
Diffusion of color in the simple exclusion process.
{\em Comm.~Pure~Appl.~Math., XLV}, 623--679.

\bibitem{starr} Starr, S.~and Conomos, M. Asymptotics of the spectral gap for the 
interchange process on large hypercubes. {\it Preprint.} 
{\tt http://front.math.ucdavis.edu/0802.1368}


\bibitem{thomas} Thomas, L.E.~(1980). Quantum Heisenberg ferromagnets
  and stochastic exclusion processes. {\em Journal of Math.~Phys.~21},
  1921-1924. 

\bibitem{wilson} Wilson, D. 
Mixing times of lozenge tiling and card shuffling Markov chains.
Ann.~Appl.~Prob., to appear. 

\bibitem{yau} Yau, Horng-Tzer (1997). 
Logarithmic Sobolev inequality for generalized simple exclusion processes.
{\em
Probab. Theory Related Fields} {\bf 109} (1997), 507--538.

\end{thebibliography}
\end{document}